\documentclass[article]{llncs}
\usepackage[usenames]{color}
\usepackage{amssymb, latexsym, mathtools,hyperref,graphicx,url}

\newcommand{\col}{\mathrm{col}}

\newcommand{\Z}{\mathbb{Z}}

\begin{document}

\title{A combinatorial approach to knot recognition}
\titlerunning{A combinatorial approach to knot recognition}

\author{Andrew Fish$^{1}$, Alexei Lisitsa$^{2}$, and David Stanovsk\'y$^{3}$\thanks{Partially supported by the GA\v CR grant 13-01832S.}}
\authorrunning{A. Fish, A. Lisitsa, D. Stanovsk\'y}

\institute{1. School of Computing, Engineering and Mathematics, University of Brighton, UK \\ 2. Department of Computer Science, The University of Liverpool, UK \\ 3. Department of Information Systems and Mathematical Modeling, International IT University, Almaty, Kazakhstan \&
Department of Algebra, Faculty of Mathematics and Physics, Charles University, Prague, Czech Republic}




\maketitle

\begin{abstract}
This is a report on our ongoing research on a combinatorial approach to knot recognition, using coloring of knots by certain algebraic objects called quandles. The aim of the paper is to summarize the mathematical theory of knot coloring in a compact, accessible manner, and to show how to use it for computational purposes. In particular, we address how to determine colorability of a knot, and propose to use SAT solving to search for colorings. The computational complexity of the problem, both in theory and in our implementation, is discussed. In the last part, we explain how coloring can be utilized in knot recognition.
\end{abstract}

\section{Introduction}\label{sec:intro}

Knot recognition is a central problem in computational knot theory. Given two \emph{knots} (closed loops in space without self-intersection), or rather their 
descriptions (such as their diagrams, Gauss codes, etc.), are they \emph{equivalent} in the sense that there is an ambient isotopy (informally, a continuous deformation of space) transforming one knot into the other?

Although it all started with Kelvin's hypothesis that different elements are differently knotted vortices of ether \cite{Sil},
most of the work on knot recognition is a typical mathematical endeavour, with little care paid to potential applications.
Nevertheless, current real life motivation to study knots exist: for example, some molecules, such as DNA, can occur knotted, and their chemical or biological properties depend on the way they are knotted (it is reported that certain antibiotics exploit topological properties of DNA). See \cite[Section 7]{A} or \cite{BF} for details and other applications of knots in biology, chemistry and physics.

Knot theory, since its beginnings in late 19th century, has used mostly geometric and topological methods, often supported by an algebraic framework (algebraic topology, homological algebra, polynomials). This is somewhat surprising, since tame knots can be considered as finite graph-like structures (see \cite{N} for more ideas in this direction); the tameness condition is imposed to rule out pathological cases. An entirely new approach appeared in foundational papers by Joyce \cite{J} and Matveev \cite{M} where certain algebraic structures, called \emph{quandles}, were introduced and used to obtain a complete invariant (up to reverse mirroring), called the knot quandle. While their motivation was to capture the essential part of the knot group, their works paved the way to introduce an extensive class of combinatorial knot invariants, based on coloring of arcs by (finite) quandles.

The present paper is an initial report on our project to turn the coloring method into a practical computational tool for knot recognition.
In the first part (Section \ref{sec:theory}), we summarize the mathematical theory necessary to understand the algorithms. The ideas are old, but dispersed around various sources; our contribution is a compact self-contained presentation of the material.
Then, we discuss how to find a coloring, computationally. A theoretical framework and the complexity of the problem is addressed in Section \ref{ssec:coloring_theory}. To search for a coloring in practice, we propose to use SAT solving. This idea seems to be new and our initial experiments are reported in detail in Section \ref{ssec:coloring_experiment} which is the core of the paper. In the last part, we explain how coloring can be utilized for knot recognition (Section~\ref{sec:recognition}), building upon earlier ideas of \cite{CESY,FL}.

Readers not familiar with basic notions of knot theory are recommended to 
take a fine crash course on Wikipedia\footnote{\url{http://en.wikipedia.org/wiki/Knot\_theory}}. 
A standard textbook reference is~\cite{A}. By an \emph{unknot} we mean any knot equivalent to the trivial knot, a simple circle in space. Informally, unknots are knots that can be untangled. Non-specialists are recommended to read the article \cite{HK} to learn what is so hard about unknotting unknots.

All knots in the text are assumed to be tame and oriented. For computational purposes, knots are presented in the form of diagrams, which are regular planar projections, so that the only singularities of the projection are transverse double points, and these bear additional information about the relative height of strands at every crossing. The size of a knot diagram then refers to the number of crossings, which determines, for instance, the size of its Gaussian code or other diagram encodings. Greek letters $\alpha,\beta,\gamma$ will be used for knot arcs, and latin letters $a,b,c$ for colors.

\section{Classical approach to knot recognition}

Mathematicians have been working on the knot recognition problem for more than a century, yet are nowhere near a satisfactory solution.
The standard approach is based on calculating \emph{invariants}, properties shared by equivalent knots. Classical invariants use various algebraic constructions to code some of the topological properties of a knot. Typical examples are polynomial invariants, for instance:
\begin{itemize}
	\item The \emph{Alexander polynomial}. It can distinguish most small knots, yet there are infinitely many knots with trivial Alexander polynomial (hence indistinguishable from the unknots). A polynomial time algorithm, fast in practice, exists (based on calculation of a determinant of a sparse matrix).
	\item The \emph{Jones polynomial} and several common generalizations of the Jones and Alexander polynomials. They tend to distinguish even more knots, yet it is not known whether they distinguish unknots. Better performance comes for a price: calculating (or even approximating) the Jones polynomial is known to be \#P-hard.
	\item The currently fashionable \emph{Khovanov homology} is the ``categorification" of the Jones polynomial. It provably distinguishes unknots, but is even harder to calculate in practice.
\end{itemize}

Among other classical methods, let us mention the \emph{knot group}, the fundamental group of the knot complement, which informally describes the ``holes" the knot creates in space. While fairly powerful in theory (for instance, it does distinguish unknots), it is not directly useful for calculations, since there is no algorithm to handle the isomorphism problem for finitely presented groups.

%
%

From what we have said so far, it would not even be clear whether knot recognition is algorithmically decidable. However, it was shown to be decidable by Haken in 1962, by an algorithm which was considered impractical for implementation. The complexity of knot recognition is an interesting and widely open problem.

Even the special case of \emph{unknot recognition} (formally, given a knot, is it equivalent to an unknot?) is not fully understood, although there has been significant progress over the past 20 years. At the moment, unknot recognition is known to be both in the NP and coNP complexity classes (coNP assuming the Generalized Riemann Hypothesis, GRH). A polynomial-time checkable certificate of unknottedness was found by Hass, Lagarias and Pippenger \cite{HLP}, using the theory of normal surfaces. A more natural certificate, a polynomial length sequence of Reidemeister moves that untangles an unknot, has been conjectured for a long time, but was only proved to exist by Lackenby \cite{L} very recently (however, it provides no clue for how to find that short sequence efficiently). On the other hand, a polynomial-time certificate for knottedness was proved to exist by Kuperberg \cite{Kup}, assuming GRH, taking the Kronheimer-Mr\'owka representation of the knot group over the group $SL(2,\mathbb C)$ and using GRH to project it to a representation over a finite field of a sufficiently small size.

\section{Combinatorial approach to knot recognition}\label{sec:theory}

A combinatorial approach to knot theory is possible, following Reidemeister, since two knots are ambient isotopic if and only if they have diagrams which differ by a finite sequence of local diagram rewrite rules called Reidemeister moves. Thus, invariants of knots can be constructed as invariants of knot diagrams which do not change under the Reidemeister moves.

\subsection{Coloring knots}

One of the first invariants ever considered was \emph{tricolorability}: color the arcs of a knot diagram with three colors in a way that, at every crossing, the three strands have either the same, or three different colors. Every knot admits three \emph{trivial colorings}, with all arcs painted in the same color.
A knot diagram is called tricolorable if it admits a non-trivial tricoloring. It is not difficult to show that tricolorability is invariant with respect to Reidemeister moves, and so it is a property of knots (not just their diagrams) which is invariant with respect to knot equivalence. More generally, the number of non-trivial tricolorings of a knot $K$, $\col_3(K)$, is an invariant.

\begin{figure}
\begin{center}
\includegraphics[scale=0.24]{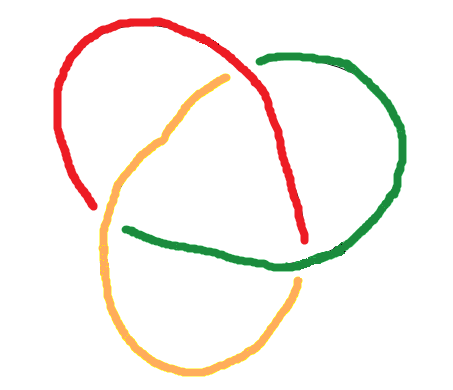}
\hskip3cm
\includegraphics[scale=0.22]{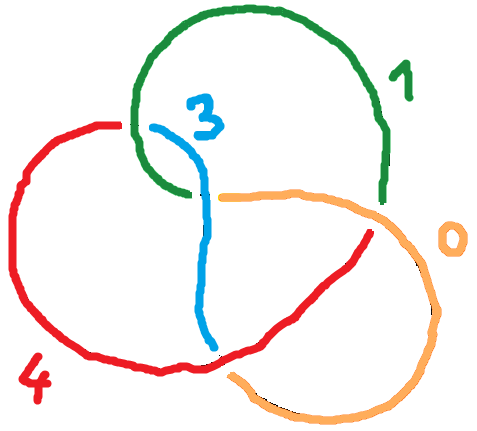}
\end{center}
\caption{Tricolored trefoil and five-colored figure-eight.}
\label{fig:coloredknots}
\end{figure}

Figure~\ref{fig:coloredknots} shows an example of a non-trivial tricoloring. Tricolorability is not a terribly good invariant: for instance, it cannot distinguish the figure-eight knot from an unknot, as the figure-eight knot admits no non-trivial tricoloring.

More generally, one can consider \emph{Fox $n$-colorings}, using colors $0,1,\dots,n-1$, and the rule that, for every crossing, the sum of the colors of the two lower strands equals twice the color of the bridge, modulo $n$. The number $\col_n(K)$ of non-trivial $n$-colorings is an invariant, which is easily calculated using linear algebra and the Chinese Remainder Theorem. Figure \ref{fig:coloredknots} shows an example of a non-trivial 5-coloring of the figure-eight knot (only four colors are used, but, importantly, the arithmetic is modulo 5). Notice that Fox 3-coloring is the same thing as tricoloring.

\begin{figure}
\begin{center}
\includegraphics[scale=0.15]{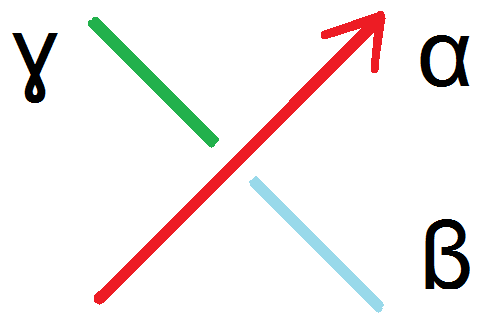}
\end{center}
\caption{Labeled crossing.}
\label{f:xing}
\end{figure}

There is an common framework for these and other arc coloring invariants. Let $C$ be a set (of colors) and $*$ a binary operation on $C$. A \emph{coloring} of a knot diagram is a mapping $f$ assigning to every arc a color from $C$ in a manner that, for every crossing with arcs labeled $\alpha,\beta,\gamma$ as in Figure \ref{f:xing}, the equation
\begin{equation}f(\alpha)*f(\beta)=f(\gamma)\tag{E}\end{equation}
is satisfied. A coloring is called trivial if it only uses one color.
For example, tricolorability uses a set $C$ with $|C|=3$ and the operation defined by $a*a=a$ for every $a\in C$ and $a*b$ equal to the third color whenever $a\neq b$. Fox $n$-coloring uses $C=\{0,1,\dots,n-1\}$ and $a*b=2a-b\bmod n$.

Let $\col_Q(D)$ denote the number of non-trivial colorings of the knot diagram $D$ by the algebraic structure $Q=(C,*)$.
Not every algebraic structure provides an invariant.
The following theorem determines such instances.

\begin{theorem}\label{thm:invariant_if_quandle}
Let $Q=(C,*)$ be an algebraic structure over a set $C$ with a binary operation $*$ satisfying the following conditions for every $a,b,c\in C$:
\begin{enumerate}
	\item[(1)] $a*a=a$ (idempotence);
	\item[(2)] there is a unique $x\in C$ such that $a*x=b$ (unique left division);
	\item[(3)] $a*(b*c)=(a*b)*(a*c)$ (left self-distributivity).
\end{enumerate}
Then $\col_Q$ is a knot invariant.
\end{theorem}

The algebraic structures satisfying conditions (1), (2), (3) are called \emph{quandles}. (We use the \emph{left} notation for quandles.) 

Theorem \ref{thm:invariant_if_quandle} is a consequence of the theory developed in \cite{J,M}. Here is a sketch of the direct proof.

\begin{figure}
\begin{center}
\includegraphics[scale=0.25]{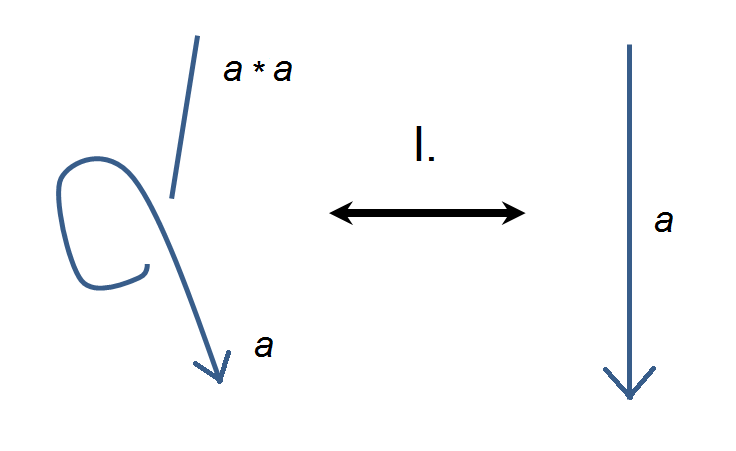}
\qquad\qquad
\includegraphics[scale=0.25]{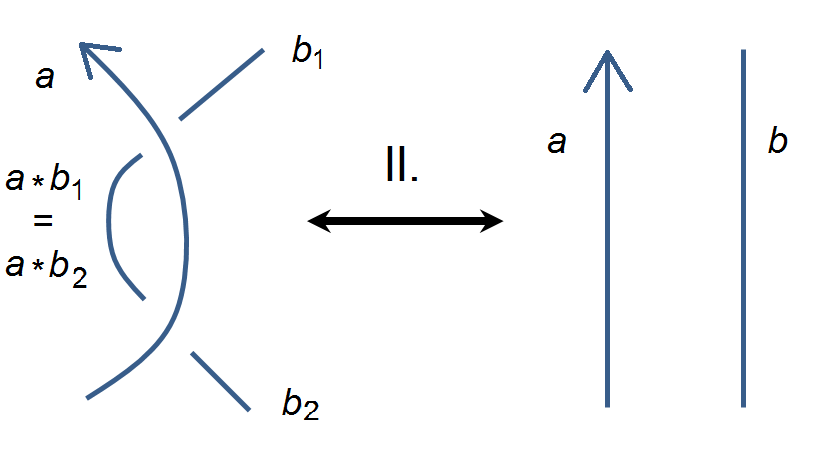}
\medskip\\
\includegraphics[scale=0.25]{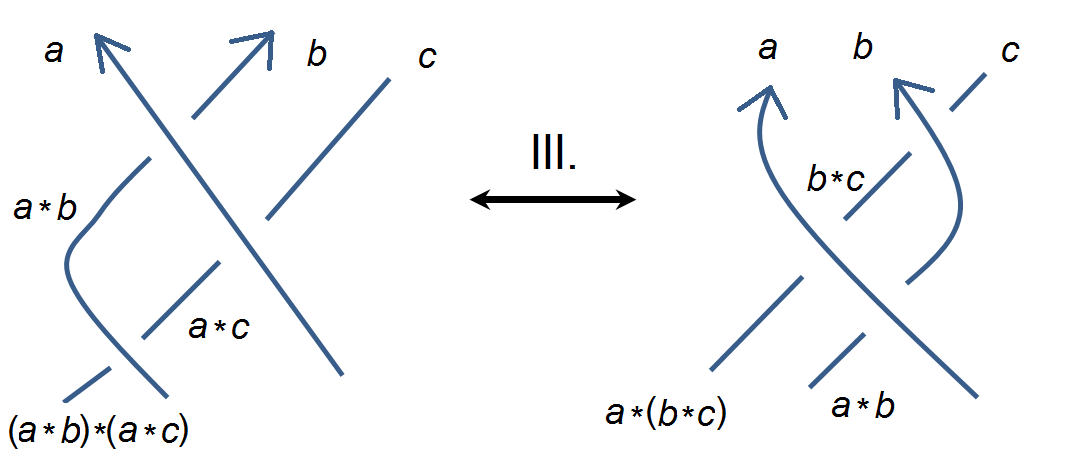}
\end{center}
\caption{Reidemeister moves and invariance of coloring.}
\label{f:rm}
\end{figure}

It is sufficient to show that $\col_Q$ is invariant with respect to the three Reidemeister moves, i.e., if two diagrams $D,D'$ differ by a single move, the number of non-trivial colorings remains the same. We prove that every coloring of $D$ corresponds uniquely to a coloring of $D'$ that uses identical colors for arcs outside the scope of the move, including those crossing the move's borderline.

For the type I move, there are two options for the orientation. For the downwards orientation as in Figure \ref{f:rm}, assume the color of the lower arc in the lefthand picture is $a$. Then, the color of the upper arc is $a*a$. Using (1), we have $a*a=a$, hence every coloring of the lefthand diagram colors both arcs in the same color, and thus corresponds to a unique coloring of the righthand diagram. For the upwards orientation, the upper arc has a color $x$ such that $a*x=a$. Using (2), this color is uniquely determined, and using (1), $x=a$, and the same argument proves the case.

For the type II move, there are two options for the orientation, again. For the upwards orientation as in Figure \ref{f:rm}, assume the colors of the border arcs in the lefthand picture are $a,b_1,b_2$, respectively. Then, the color of the middle arc is $a*b_1=a*b_2$, hence uniquely determined by $a,b_1,b_2$. Moreover, using (2), we have $b_1=b_2$. Hence every coloring of the lefthand diagram corresponds to a uniquely determined coloring of the righthand diagram. Thanks to (1), non-trivial colorings correspond mutually. A similar argument proves the downwards orientation.

For the type III move as in Figure \ref{f:rm}, assume the colors of the top arcs be $a,b,c$, respectively. The colors of the remaining arcs are uniquely determined as in the picture. To obtain the same color for the left bottom arc, we use condition (3). For the other three orientations, we proceed similarly;  condition (2) is needed for unique solutions as in case I.

Looking at the proof sketch, we see that the three axioms of quandles relate very well to the three Reidemeister moves. Invariance with respect to type II moves uses precisely condition (2), and invariance with respect to types I and III moves uses precisely conditions (1) and (3), respectively, in some cases under the assumption of unique left division. The converse statement, that \emph{quandles are the only structures providing arc coloring invariants of knots}, is also true in certain way, but the precise statement and a proof is beyond the scope of the present paper.

\subsection{Quandles}\label{ssec:quandles}

A natural question arises: what really are quandles? The answer is not yet fully understood, but in the most important case, of so called \emph{algebraically connected quandles}, it is. We will briefly summarize the theory in the next paragraph. Readers with no background in group theory can safely skip it and continue reading about the examples of quandles below.

A quandle $Q=(C,*)$ is algebraically connected if the permutation group generated by the left translations $L_a(x)=a*x$ is transitive on $C$. It is not difficult to prove that colors used in a coloring generate an algebraically connected subquandle \cite[Section 5.2]{Oht}, hence, without loss of generality, we can consider only $Q$-colorings with a connected $Q$. Such quandles were studied in detail in several papers, the strongest results were obtained recently in \cite{HSV}. The main theorem states that connected quandles are in 1-1 correspondence to certain configurations in transitive groups: every connected quandle on a set $C$ is uniquely determined by a pair $(G,\zeta)$ where $G$ is a transitive group on $C$ and $\zeta$ is a central element of the stabilizer $G_e$ such that $\langle\zeta^G\rangle=G$ (uniqueness up to the choice of $e$); an isomorphism theorem is available. Using the theory of \cite{HSV}, and a library of transitive groups of degree $n$, one can easily enumerate all connected quandles of size $n$ up to isomorphism. Currently, such a library is available for $n\leq47$.

Let us introduce three important classes of quandles.

A quandle $Q=(C,*)$ is called \emph{affine} if there is an abelian group $G=(C,+)$ (informally, an addition on the set of colors) and an automorphism $\varphi$ such that $a*b=a-\varphi(a)+\varphi(b)$ for every $a,b\in C$ (informally, $*$ is an affine combination of colors). For example, Fox $n$-coloring uses the affine quandle with $G=\Z_n$ and $\varphi(x)=-x$. An affine quandle is connected if and only if the mapping $x\mapsto x-\varphi(x)$ is onto.

Given a group $G$ and a conjugacy class $C$ in $G$, then $(C,*)$ with $a*b=aba^{-1}$ is called a \emph{conjugation quandle} over $G$. Conjugation quandles may or may not be connected.

A quandle is called \emph{simple} if it has no proper homomorphic images (other than itself and the trivial one). Finite simple quandles with $|C|>2$ are always connected and were characterized in \cite{AG,J2}: they are either affine, or arise from a finite simple group using a sort of parametrized conjugation operation on its orbit of transitivity; the construction can be used to create a large list of simple quandles. Within the framework of \cite{HSV} (see the description above), simple quandles are recognized as those where every factor of $G$ is cyclic and the centralizer of $\zeta$ is contained in $G_e$ \cite{AG}; the message is that it is computationally easy to verify whether a given quandle is simple.

\subsection{Colorability and simple quandles}

If $\col_Q(K)>0$, we say that the knot $K$ is \emph{$Q$-colorable}. Now, $Q$-colorability is also an invariant, which is easier to calculate than $\col_Q(K)$: one only has to establish that $\col_Q(K)>0$, whereas the actual number of colorings, $\col_Q(K)$, is often rather big due to the high symmetry of quandles \cite{CESY}, and finding all of them can be computationally costly.

The following observation suggests that if we only care about colorability (and not the actual number of colorings), we can restrict to simple quandles.

\begin{lemma}\label{lm:simple}
Let $K$ be a knot which is colorable by a finite quandle $Q$. Then $K$ is colorable by a finite simple quandle $Q'$ such that $|Q'|\leq|Q|$.
\end{lemma}

\begin{proof}
Consider a non-trivial coloring $f$ of $K$ by $Q$.
Let $Q''$ be the subquandle of $Q$ generated by all colors used in the coloring $f$. Indeed, $f$ is a coloring of $K$ by $Q''$. Let $\alpha$ be a maximal congruence of $Q''$. Then the factor, $Q'=Q''/\alpha$, is a finite simple quandle (because $\alpha$ is maximal). Consider the composition $f'=\pi\circ f$ where $\pi$ is the natural projection $Q''\to Q'$. Then $f'$ is a non-trivial coloring of $K$ by $Q'$. Equation (E) follows from the fact that $\pi$ is a homomorphism: for every crossing as in Figure \ref{f:xing}, $f'(\alpha)*f'(\beta)=\pi(f(\alpha))*\pi(f(\beta))=\pi(f(\alpha)*f(\beta))=\pi(f(\gamma))=f'(\gamma)$. If $f'$ was trivial, then all colors used by $f$ were in one block, $B$, of $\alpha$. Since congruence blocks are subquandles and $Q''$ is generated by all such colors, $B=Q''$, hence $\alpha$ was the total congruence $Q''\times Q''$, a contradiction.
\end{proof}

\section{Finding a coloring}\label{sec:coloring}

Given a quandle $Q$ and a knot $K$, a natural question is ``how do we calculate $\col_Q(K)$, the number of non-trivial $Q$-colorings?" For simplicity, let us focus on the decision problem of \emph{$Q$-colorability}: ``does there exists a non-trivial $Q$-coloring of $K$?", i.e. can we decide if $\col_Q(K)>0$. We will always assume a knot $K$ is presented as a diagram, and we let $|K|$ denote the number of crossings of the diagram.

\subsection{Theory}\label{ssec:coloring_theory}

Naturally, this problem is in the complexity class NP: given an assignment of colors to arcs, it is easy to check whether it is a non-trivial $Q$-coloring. To find a coloring, one has to solve a system of equations over the quandle $Q$: for every crossing as in Figure \ref{f:xing}, we have the equation $x_\alpha*x_\beta=x_\gamma$, where $x_\alpha,x_\beta,x_\gamma$ are variables that determine the colors of the arcs $\alpha,\beta,\gamma$. This can be viewed as an instance of the Constraint Satisfaction Problem (CSP): arcs are variables, $Q$ is the domain, and the equations are the constraints. In general, CSP is an NP-complete problem (SAT is a special instance of CSP). But how hard are the knot coloring instances? In particular, how hard is $Q$-colorability for a fixed quandle $Q$? How hard is $Q$-colorability for a fixed knot $K$?

One can try a simple brute force search: to every one of $|K|$ arcs, assign one of $|Q|$ colors, and check whether this is a non-trivial coloring; we are  handling $|Q|^{|K|}$ assignments. For a fixed knot $K$, this gives a polynomial time algorithm with respect to $|Q|$, but this is impractical for larger knots. Can we do with a better exponent than $|K|$?

Given a braid representation of $K$, with $n$ strands (the smallest possible $n$ is called the \emph{braid index} of $K$), the brute force search is only required to assign colors to the initial $n$ arcs, since the colors of the remaining arcs are uniquely determined; hence, we only handle $|Q|^n$ assignments. Whilst this provides a dramatic improvement in practice, the algorithm is still exponential-time with respect to $|K|$ (leaving aside conversion into braid representation). See \cite{CESY} for a more detailed description and experimental results.

Universal algebra provides a different point of view on coloring. We are solving a system of $|K|$ equations over the quandle $Q$. The complexity of equation solving over general algebraic structures has been studied extensively. In our case, it follows from general results of Larose and Z\'adori \cite{LZ} that the problem of solving a general system of equations over a quandle $Q$ is polynomial-time if and only if $Q$ is affine, and it is NP-complete otherwise. Consequently, colorability by affine quandles (and Fox coloring in particular) is easy. Nevertheless, it does not mean that colorability by non-affine quandles is hard, because the systems of equations arising from knot coloring have a very special form! It is an interesting open problem whether a polynomial time algorithm for $Q$-colorability exists for every quandle $Q$.

We are not aware of any other theoretical results applicable to the complexity of quandle colorings.

\subsection{Practice}\label{ssec:coloring_experiment}

The only attempt to calculate quandle colorings on a larger scale we know about is reported in \cite{CESY}. Their implementation was tested on all knots with at most 12 crossings, 
their performance graph exhibits a fast growth of running time with respect to $|Q|$.

We propose a more efficient approach to knot coloring: a reduction to the Boolean Satisfiability Problem, SAT.
Fix a connected quandle $Q=(\{1,\dots,q\},*)$ and a knot diagram $K$ with $|K|=n$, with arcs numbered $\alpha_1,\dots,\alpha_n$. We consider $nq$ boolean variables $v_{i,c}$ that determine whether the arc $\alpha_i$ has the color $c$. We need to satisfy the following constraints:
\begin{itemize}
	\item Every arc has a unique color: the obvious description uses the clauses
	\[v_{i,1}\vee\ldots\vee v_{i,q}\quad\text{ and }\quad \neg v_{i,c}\vee\neg v_{i,d}\]
	 for every $i=1,\dots,n$ and $c=1,\dots,q$, $d=c+1,\dots,q$.
	\item Not all arcs have the same color: the obvious description uses, for every $c=1,\dots,q$, the clause
	\[\neg v_{1,c}\vee\ldots\vee \neg v_{n,c}.\]
	\item Crossing equations (E): for every crossing, with a bridge $\alpha_i$ over the strands $\alpha_j$ to the right and $\alpha_k$ to the left, we use $q^2$ formulas of the form \[(v_{i,c}\wedge v_{j,d})\to v_{k,c*d}.\]
\end{itemize}

An important ingredient is basic symmetry breaking: since connected quandles are homogeneous, we can assume that the arc $\alpha_1$ has color 1. (If we were counting the number of colorings, and not just checking colorability, then we could recover $\col_Q(K)$ by multiplying the answer by $|Q|$).

We implemented the procedure, employing a standard SAT-solver MiniSat 1.14. 
For our experiments, we used the family
\begin{itemize}
	\item[{\bf SQ}.] of all 354 simple quandles of size $\leq47$, indexed in accordance to size,
\end{itemize}
(the reason is explained in Lemma \ref{lm:simple}); and the following three families of knots:
\begin{itemize}
\item[{\bf 12A}.] all 1288 alternating knots with crossing number 12 (taken from \cite{knotinfo}),
\item[{\bf T2}.] $(2,n)$-torus knots with $n=11,21,31,\dots,91$,
\item[{\bf T3}.] $(3,n)$-torus knots with $n=14,20,26,\dots,98$.
\end{itemize}
(Torus knots are some of the best known and easiest to implement, potentially infinite, families of knots.)
Figures \ref{f:12}-\ref{f:qn_torus} present performance graphs with respect to various parameters.

\begin{figure}
\begin{center}
\includegraphics[scale=0.7]{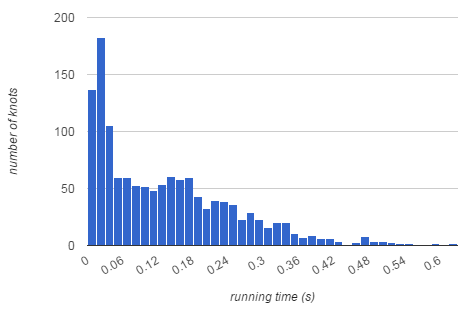}
\end{center}
\caption{Distribution of running times (in seconds) for the {\bf 12A} family, averaged over all quandles in {\bf SQ}.}
\label{f:12}
\end{figure}

First of all, we must stress that neither the number of crossings, nor the braid index, are the most relevant knot parameters with respect to running time. Figure \ref{f:12} shows the distribution of running times for the family {\bf 12A}. Whilst most knots are colored quickly, there are several ``slow" knots, with the record being held by the knot with KnotInfo \cite{knotinfo} code 12a\_1092, see Figure \ref{f:1092}. We were unable to identify any possible reason that makes this particular knot harder than others. Also notice the high noisiness of the performance graphs for the $(3,n)$-torus knots in Figure \ref{f:k}. Yet again, it is not merely the number of crossings that determines the running time.

For the {\bf T3} family, the average running time fluctuates more than the median running time. This suggests that the peaks are caused by a few very long computations, while a ``random" coloring job runs relatively fast. For example, for the $(3,92)$-torus knot, which is the last but one node on the right graph in Figure~\ref{f:k}, the running times for most quandles are no more than a few seconds, but the few exceptions run for a few hundreds of seconds. Perhaps this is an indication that while it is fast to determine $Q$-colorability on average, it could be hard in the worst case.

\begin{figure}
\begin{center}
\includegraphics[scale=0.5]{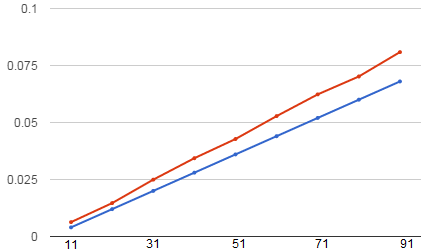}
\quad
\includegraphics[scale=0.5]{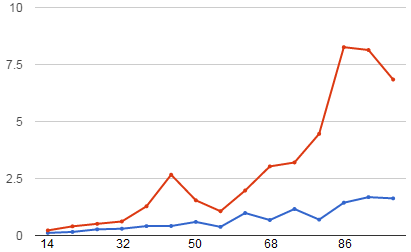}
\end{center}
\caption{Average (red, top line) and median (blue, bottom line) running times (s) over all quandles in {\bf SQ}, with respect to knot size for knots in {\bf T2} (on the left) and {\bf T3} (on the right), respectively.}
\label{f:k}
\end{figure}

Both graphs in Figure \ref{f:k} suggest that the dependence of median running time with respect to knot size is roughly linear for torus knots.
This is not surprising, and probably far from the general case because $(k,n)$-torus knots have a fixed braid index $k$, for every $n$, and so a brute force search in the spirit of \cite{CESY} runs in linear time with respect to $n$, by choosing $k$ initial colors and running along the torus producing the remaining coloring values. Nevertheless, in the {\bf T3} case, it seems the SAT solver has a hard time to recognize the braids, and therefore, the running times fluctuate.

\begin{figure}
\begin{center}
\includegraphics[scale=0.5]{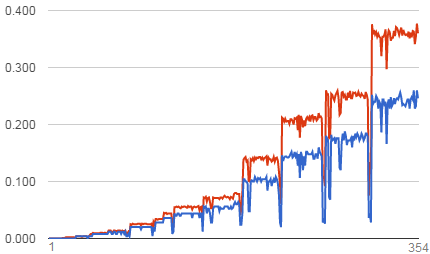}
\quad
\includegraphics[scale=0.5]{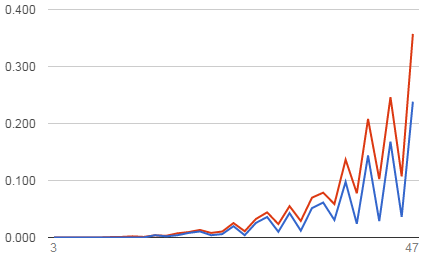}
\end{center}
\caption{Average (red, top line) and median (blue, bottom line) running times (s) over knots in {\bf 12A}, with respect to quandle index (on the left), and quandle size (on the right), respectively.}
\label{f:q12}
\end{figure}

\begin{figure}
\begin{center}
\includegraphics[scale=0.5]{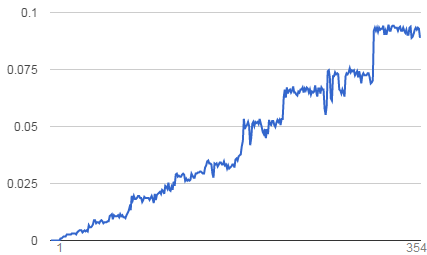}
\quad
\includegraphics[scale=0.5]{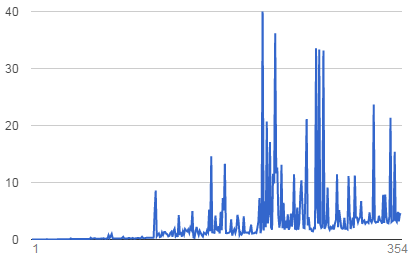}
\end{center}
\caption{Average running times (s) over knots in {\bf T2} (on the left) and {\bf T3} (on the right), respectively, with respect to quandle index.}
\label{f:qn_torus}
\end{figure}

The graphs in Figures \ref{f:q12} and \ref{f:qn_torus} indicate the time complexity with respect to quandle index and quandle size (by index, we mean the position in the list {\bf SQ}). Looking at the lefthand graphs, addressing the {\bf 12A} and {\bf T2} knots, we observe a staircase behavior. It tells that quandles of equal size have similar running times. For {\bf 12A} knots, we observe occasional deep drops. These are precisely the positions of non-affine quandles, which are, on average, much faster than the affine ones. This is somewhat surprising, since, in theory, non-affine quandles are the harder ones (see Section \ref{ssec:coloring_theory}). The reason might be a higher symmetry of affine quandles, invisible to MiniSat, together with its inability to employ the fast coloring algorithm.

For {\bf T3} knots (Figure \ref{f:qn_torus}, right), the situation is different: the graph is much noisier and the majority of the peaks are at non-affine positions. While non-affine quandles seem faster in ``random" cases, they also provide the majority of very long running times.

Observing the growth of running time with respect to quandle index, we see it is somewhat faster than linear, but certainly subquadratic. What about the dependence with respect to quandle size, $|Q|$? Since the number of simple quandles of size $n$ appears to grow roughly linearly with $n$ (this is only an empirical observation based on the family {\bf SQ}, we have no formal arguments yet), we shall add one to the degree. Figure \ref{f:q12} confirms the idea. Since all affine simple quandles have prime power size, and all non-affine simple quandles have size divisible by at least two primes, the right graph reflects the drops and peaks of the left graph.

We can draw the following conclusion. Even a simple implementation, using an obsolete solver and no fine tuning, is competitive, outperforming the brute-force search considerably.
Our programs and computation data are available at our website\footnote{\url{http://www.karlin.mff.cuni.cz/~stanovsk/quandles}}.

\section{Recognizing knots and unknots}\label{sec:recognition}


\subsection{Proving inequivalence}

Let $K_1,K_2$ be two knots. If we find a quandle $Q$ such that $\col_Q(K_1)\neq\col_Q(K_2)$ (or, in particular, such that $K_1$ is $Q$-colorable and $K_2$ is not), we can conclude that $K_1$ and $K_2$ are not equivalent. Large libraries of connected quandles can be created using the methods described in Section \ref{ssec:quandles}. For instance, it is feasible to compile the complete list of connected quandles up to size 47, the complete list of conjugation quandles over any group with hundreds of elements, the complete list of affine quandles of any size in the order of thousands.

An interesting experiment is described in \cite{CESY}. All knots with at most 12 crossings can be distinguished (up to mirror image) using a list of 26 finite quandles, the largest having 182 elements. Most pairs of knots are distinguished using a fairly small quandle, and therefore relatively quickly. More details and other interesting findings can be found in \cite{CESY}.

Unfortunately, no upper bound is known for the size of a quandle that distinguishes two inequivalent knots, and it is not even known whether a finite quandle is sufficient. In particular, it is not known whether coloring provides a decision procedure for knot recognition.

\subsection{Recognizing unknots}\label{ssec:unknot}

In \cite{FL}, the first two named authors proposed to use automated theorem proving on knot quandles for unknot recognition, i.e., for the decision of whether a given knot is equivalent to an unknot. Here we present the idea in terms a quandle coloring in a somewhat reshaped manner.  The method is based on Theorem \ref{thm:unknot_col}, which is collected from several works.

\begin{theorem}\label{thm:unknot_col}
The following are equivalent for a knot $K$:
\begin{enumerate}
	\item[(1)] $K$ is knotted (i.e., not an unknot).
	\item[(2)] There is a quandle $Q$ such that $\col_Q(K)>0$.
	\item[(3)] There is a finite quandle $Q$ such that $\col_Q(K)>0$.
	\item[(4)] There is a finite simple quandle $Q$ such that $\col_Q(K)>0$.
	\item[(5)] There is a conjugation quandle $Q$ over the group $SL(2,p)$, for a prime $p$, such that $\col_Q(K)>0$.
\end{enumerate}
\end{theorem}

The equivalence of (1) and (2) was proved in \cite{J,M}. The equivalence of (2) and (3) was proved in \cite{CSV}. The equivalence of (3) and (4) follows from Lemma~\ref{lm:simple}. The equivalence of (1) and (5) is a translation of Kuperberg's certificate of knottedness \cite{Kup} into the language of quandles (this part of Kuperberg's result is independent of GRH which is only needed to guarantee that $p$ is small enough).
Notice that the implication $(2)\Rightarrow(1)$ follows from Theorem \ref{thm:invariant_if_quandle}: the trivial knot has only one arc, hence admits only trivial colorings; since $\col_Q$ is an invariant, so does any unknot.

Theorem \ref{thm:unknot_col} turns quandle coloring into a decision procedure for unknot recognition: given a knot on input, we run two independent algorithms (ideally in parallel), one seeking for a certificate of knottedness, the other for unknotting. For knottedness, it is sufficient to find a non-trivial coloring; condition (3) says it is sufficient to consider finite colorings, and conditions (4) and (5) suggest particular families of finite quandles to try. For unknottedness, it is sufficient to \emph{prove} that no non-trivial coloring exists, and automated theorem proving provides a tool to do so.

Detailed experiments with the proposed decision procedure will be reported in a subsequent paper.
To certify knottedness, for example, the family {\bf SQ} of simple quandles of size $\leq47$ is sufficient for all {\bf 12A} knots and for many larger knots. Certification running times are competitive against state-of-the-art in knot recognition.
The unknottedness certification has been tested on various famous ``hard unknots" with a positive outcome, see \cite{FL} for details. 

\begin{figure}
\begin{center}
\includegraphics[scale=0.33]{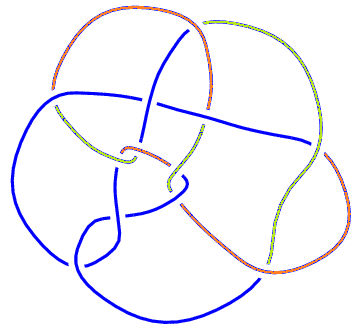}
\end{center}
\caption{Tricolored knot with KnotInfo \cite{knotinfo} code 12a\_1092.}
\label{f:1092}
\end{figure}

Efficience of the knottedness certification heavily relies on efficience of the coloring procedure used. Nevertheless, for a particular knot, the complexity of the two problems seems to be related only loosely: for example, the slowest knot 
reported in Section \ref{ssec:coloring_experiment} is tricolorable (see Figure \ref{f:1092}), hence certified very quickly.

\subsection{Classic vs. combinatorial}

We believe that many classical results in knot theory can be explained in terms of coloring, perhaps in a more systematic or conceptually cleaner form. A few sample results can be found in \cite[Section 7]{CESY}.
A different kind of example has been mentioned earlier in our paper: Kuperberg's certificate of knottedness is a prime $p$ 
and a non-commutative representation of the knot group over $SL(2,p)$, which is the same thing as a non-trivial coloring by a conjugation quandle over $SL(2,p)$ (the Wirtinger presentation of the knot group uses arcs as generators and conjugation relations analogous to condition (E) for every crossing).

Here is another, perhaps more important, example. The Alexander invariant
is essentially the same thing as coloring by affine quandles. We refer to \cite{B} for a precise statement, and we state formally a result
related to Theorem~\ref{thm:unknot_col}.

\begin{theorem}\label{thm:alex_col}
The following are equivalent for a knot $K$.
\begin{enumerate}
	\item[(1)] $K$ has a non-trivial Alexander polynomial.
	\item[(3)] There is a finite affine quandle $Q$ such that $\col_Q(K)>0$.
	\item[(4)] There is a finite simple affine quandle $Q$ such that $\col_Q(K)>0$.
\end{enumerate}
\end{theorem}

The theorem can be used to considerably improve efficiency of search for inequivalence certificates. Let $K_1,K_2$ be two knots. First, we calculate their Alexander polynomials $\Delta_1,\Delta_2$ (a fast algorithm exists).
If $\Delta_1\neq\Delta_2$, the knots are inequivalent. Else, continue as before. If we only test for $Q$-colorability (such as in the knottedness certificate described in the previous subsection), we can exclude all affine quandles from the search. This can result in a considerable cut of the search space: for instance, out of the 354 simple quandles with at most 47 elements, only 23 are non-affine.


\section{Conclusion}\label{sec:conclusion}

We presented a new approach to knot recognition, based on coloring the arcs of a knot by certain algebraic objects. Unlike most of the classical invariants, finding a coloring is a combinatorial problem, amenable to a (smart) exhaustive search. The coloring method is provably a decision procedure for recognizing unknots, and works very well in practice for general knot recognition, too. In both cases, most instances of inequivalence are certified very quickly.


The obvious future work is a careful analysis of the parameters of the decision procedure (efficient SAT encoding, choice of the quandle family, heuristics in the spirit of Theorem \ref{thm:alex_col}) to improve efficience of knottedness certification in practice.
To extend the method towards general knot recognition, a generalization of Theorem \ref{thm:unknot_col} is much needed.

\end{document}